\documentclass{amsart}%
\usepackage{amsfonts}
\usepackage{amsmath}
\usepackage{amssymb}
\usepackage{graphicx}%
\providecommand{\U}[1]{\protect\rule{.1in}{.1in}}
\newtheorem{theorem}{Theorem}
\theoremstyle{plain}

\newtheorem{definition}{Definition}
\newtheorem{example}{Example}

\newtheorem{lemma}{Lemma}

\newtheorem{proposition}{Proposition}
\newtheorem{remark}{Remark}

\numberwithin{equation}{section}
\begin{document}
\title[Moment sequences]{Moment sequences and difference equations.}
\author{Pawe\l \ J. Szab\l owski}
\address{emeritus in Department of Mathematics and Information Sciences, \\
Warsaw University of Technology\\
ul Koszykowa 75, 00-662 Warsaw, Poland\\
 }
\email{pawel.szablowski@gmail.com}
\thanks{The author is very grateful to the unknown referee for his (or her) helpful,
extensive and friendly remarks and leads.}
\date{March 2022}

\begin{abstract}
We recall the definition and properties of a moment sequence and show that all
real sequences whose Hankel matrices have finite rank (see definition in the
sequel) satisfy a homogeneous linear equation with constant coefficients. Then
we analyze the cases in which a difference equation with constant coefficients
and suitably chosen initial conditions and having as an input a positive
moment sequence has a solution that is a positive moment sequence. We give one
general simple result and give many examples illustrating the theory. The main
result states that the roots of the odd multiplicity of the characteristic
equation must lie outside the support of the measure that produces the moment
sequence that is in the input and the initial conditions suitably chosen.

\end{abstract}
\maketitle

\section{Introduction}

The purpose of this note is to examine the situation when there exists an when
there is an overlap between a certain important class of real sequences and
homogeneous difference equations with constant coefficients and also the cases
when a difference equation with a positive moment (pm) sequence as input
produces a positive moment sequence. It will turn out that firstly, all
sequences having a finite rank of their so-called Hankel matrices satisfy a
certain homogeneous difference equation with constant coefficients and
secondly, that not every pm sequence produces a pm sequence in this way. This
depends on whether the roots of the characteristic equation lie in the support
of the measure that generates the input pm sequence and also on the fact that
the initial conditions are suitably chosen. In this way, one can use the
properties of the solutions of the difference equation with given pm input to
test whether the roots of the characteristic equation lie in the support of
the measure that produced the pm input sequence.

Another purpose of this note is to provide a probabilistic interpretation of
the results presented in the excellent paper of Bennett \cite{Benn11} and thus
present their sometimes much simpler proofs. We also aim to generalize some of
these results and also provide a connection between moment sequences and
difference equations.

The paper is organised as follows. In the next section, we recall the
definition of a moment sequence and formulate sufficient conditions for a
moment sequence to identify the measures that generate it. Then we analyse the
set of positive moment (pm) sequences. We give several examples and formulate
a series of operations on pm sequences that result in another pm sequence.
Many results of this section are known and scattered throughout the
literature; we recall them for the sake of completeness of the paper and
provide, in cases when it is possible, the new proofs stemming mostly from
probability. The following section presents new results on when a difference
equation with constant coefficients and a pm input produces a pm sequence. The
final section presents several important examples.

\section{Moment sequences}

Let us agree that whenever the domain of integration is not specified
explicitly, it is assumed to coincide with the support of the measure with
respect to which the integral is taken. 

Let us start with the following definition of the moment sequence.

Let $\mu$ be a certain measure on the real line. Let $\mu^{+}$ and $\mu^{-}$
be the elements of the Jordan-Hahn decomposition of $\mu$, i.e., that $\mu
^{+}$and $\mu^{-}$ are two positive measures such that $\mu(B)\allowbreak
=\allowbreak\mu^{+}(B)\allowbreak-\allowbreak\mu^{-}(B)$ for every $\mu$-
measurable subset $B$ of the real line. Let us denote $\left\vert
\mu\right\vert $ the sum of the two positive measures $\mu^{+}$ and $\mu^{-}$.
Let us consider a (signed in general) measure $\mu$ such that:
\begin{equation}
\sup_{n\geq1}\int\left\vert x\right\vert ^{n}\left\vert \mu\right\vert
(dx)<\infty. \label{*}%
\end{equation}
Then a sequence $\left\{  m_{n}\right\}  _{n\geq0}$ of reals defined by:
\begin{equation}
m_{n}\allowbreak=\allowbreak\int x^{n}\mu(dx), \label{sB}%
\end{equation}
is called a moment sequence of the measure $\mu.$ Boas theorem (see
\cite{Boas39}) states that every sequence is a moment sequence of some signed
measure satisfying (\ref{*}). In fact, Boas's theorem asserts some more
precise knowledge of the signed measure $\mu$. Namely, we have:

\begin{theorem}
[Boas]Any real sequence of numbers $s\allowbreak=\allowbreak\left\{
s_{j}\right\}  _{j\geq0}$ can be represented in the form%
\[
s_{n}\allowbreak=\allowbreak\int_{0}^{\infty}x^{n}\mu(dx),
\]
$n\allowbreak=\allowbreak0,1,\ldots$ with
\[
\int_{0}^{\infty}\left\vert \mu\right\vert (dx)<\infty.
\]

\end{theorem}

Hence, for any real sequence $s:=\left\{  s_{j}\right\}  _{j\geq0}$ we can
define the measure $\mu$ that generates this sequence according to (\ref{sB}).
Any measure (not necessarily defined uniquely) generating sequence $s$ will be
called generating measure of this sequence, briefly $gm(s)$.

Further, together with generating measure let us define a sequence of Hankel
matrices $\mathcal{H}_{n}\allowbreak\overset{df}{=}\allowbreak\lbrack
s_{i+j}]_{i,j=0}^{n}$ and determinants of these matrices $D_{n}\allowbreak
\overset{df}{=}\allowbreak\det\mathcal{H}_{n}$.

\begin{definition}
Let $\left\{  a_{i}\right\}  _{i\geq0}$ be a real sequence. The sequence of
Hankel matrices of this sequence i.e. $\left\{  \mathcal{H}_{n}\right\}
_{n\geq0}$ will be called sequence of the Hankel matrices of the sequence
(Hms). The sequence of determinants of the Hms , i.e., $\left\{
D_{n}\right\}  _{n\geq0}$ is called Hankel transform of the sequence $\left\{
a_{n}\right\}  .$
\end{definition}

Let us note, that there can be many real sequences having the same Hankel
transform. The in-depth analysis of the relationship between a real sequence
and its Hankel transform was done e.g. in \cite{Nash15}(chapter 5). For our
purposes it will be enough, to notice that following the Theorem of Kronecker
for a sequence $\left\{  a_{n}\right\}  _{n\geq0}$ with the sequence of Hms
$\left\{  \mathcal{H}_{n}\right\}  _{n\geq0}$ and Hankel transform $\left\{
D_{n}\right\}  _{n\geq0}$ to satisfy condition%
\begin{equation}
\sup_{n\geq1}~rank(\mathcal{H}_{n})=r<\infty\label{fR}%
\end{equation}
it is necessary and sufficient that $D_{r-1}\neq0$ and $\forall n\geq r$ :
$D_{n}\allowbreak=\allowbreak0.$

Sequences $s$ satisfying condition (\ref{fR}) will be called of finite rank
briefly fR-sequences.

\begin{remark}
Notice that, if the support of measure $\mu_{X}$ is finite say consisting of
points $\left\{  \alpha_{i}\right\}  _{n=1}^{m}$ with masses $\left\{
p_{i}\right\}  _{i=1}^{m},$ then each moment $m_{n}$ has the following form :%
\[
m_{n}=EX^{n}=\sum_{i=1}^{m}p_{i}\alpha_{i}^{n},
\]
and thus the sequence of Hms has the following form%
\[
\lbrack EX^{i+j}]_{i,j=0}^{n}=\sum_{k=1}^{m}p_{k}[\alpha_{k}^{i+j}%
]_{i,j=0}^{n}.
\]
Now notice, that each matrix of the form $[\alpha_{k}^{i+j}]_{i,j=0,\ldots,n}$
has rank $1,$ hence for every $n\geq m$ matrix $[EX^{i+j}]_{i,j=0,\ldots,n}$
has at most rank $m.$ Thus the sequences of such, finitely supported measures,
have their Hankel transforms consisting of zeros for elements with indices
greater than $m.$ In other words, we have shown that all moment sequences of
finitely supported measure are fR-sequences.
\end{remark}

\begin{proposition}
\label{FS}Let $s$ be a real sequence. Assume that its $gm(s)$ is finitely
supported and the cardinality of the support is $r,$ then there exist $r$
constants $b_{0},\ldots,b_{r-1}$ such that for $m\geq0$ :
\begin{equation}
\sum_{k=0}^{r-1}b_{k}s_{k+m}=s_{r+m}. \label{eqfR}%
\end{equation}
In other words all elements of sequence $s$ satisfy a homogeneous linear
equation with constant coefficients (see the definition and properties below
in Section \ref{MD}).
\end{proposition}

\begin{proof}
Assume that the support of $gm(s),$ that we will denote by $\mu,$ consists of
points $x_{1},\ldots,x_{r}\in\mathbb{R}$ . Let coefficients $b_{i},$
$i\allowbreak=\allowbreak0,\ldots,r-1$ be defined by the relationship.%
\[
\prod_{j=1}^{r}(x-x_{j})\allowbreak=\allowbreak x^{r}-\sum_{j=0}^{r-1}%
b_{j}x^{j}.
\]
Now notice that the measure $v$ defined by:
\[
v(dx)=\prod_{j=1}^{r}(x-x_{j})\mu(dx)
\]
is the zero measure. Hence, in particular, we have $\forall m\geq0$: $\int
x^{m}v(dx)\allowbreak=\allowbreak0.$ But this means that equation (\ref{eqfR})
is satisfied by the sequence $s$.
\end{proof}

The converse statement also holds. Namely, we have the following theorem that
has been formulated and proved in \cite{Nash15}(Chapter 5 Thm. 2.3). This
theorem can be traced back to Kronecker (from 1881)..

\begin{theorem}
Rank of Hankel transform of a real sequence $s$ is finite iff its generating
measure $\mu(dx)$ satisfies the following condition: There exist a polynomial
$Q(x)$ such that the measure $v(dx)\allowbreak=\allowbreak Q(x)\mu(dx)$
satisfies the following condition: $\forall m\geq0:$%
\begin{equation}
\int x^{m}v(dx)=0. \label{nmea}%
\end{equation}
More precisely, let the rank in question be $r,$ then there exist $r$ numbers
$b_{0},\ldots,b_{r-1}$ such that for $m\geq0$ :
\[
\sum_{k=0}^{r-1}b_{k}s_{k+m}=s_{r+m}.
\]
In other words there exists a homogeneous difference equation of order $r$
such that all elements of $s$ with indices greater than $r-2$ are defined by
this equation with initial conditions $s_{0},\ldots,s_{r-1}$. Moreover,
coefficients $\left\{  b_{i}\right\}  _{i=0}^{r-1}$ are defined by the
elements $s_{r},\ldots,s_{2r-1}$ of the sequence $s.$ In other words the
sequence is uniquely defined by its first $2r$ values. Besides, denoting
$Q(x)\allowbreak=\allowbreak x^{r}-\sum_{j=0}^{r-1}b_{j}x^{j}$ and
$v(dx)\allowbreak=\allowbreak Q(x)\mu\left(  dx\right)  $ we see that the
condition (\ref{nmea}) is satisfied.
\end{theorem}

\begin{remark}
Notice that if we assume that $gm(s)$ is identifiable by moments, then the
condition (\ref{nmea}) means that the measure $Q(x)\mu(dx)$ is the zero
measure and iff polynomial $Q(x)$ has all different and real roots, then from
Proposition \ref{FS} it follows that the measure is finitely supported and
their support consists of the root of the polynomial $Q$.
\end{remark}

From now on we will be interested in the cases when measure $\mu$ is a
positive measure, that is $\mu^{-}\allowbreak=\allowbreak0$. The exhaustive
study of conditions guaranteeing identifiability by moments of the positive
measure was done recently in \cite{Gwo2017}.

Let $\operatorname*{supp}\mu$ denote its support and then let us define a
sequence of numbers defined in the following way:%
\[
m_{n}\allowbreak=\allowbreak\int_{\operatorname*{supp}\mu}x^{n}\mu(dx),
\]
for $n\geq0.$ We will call such sequence a p(ositive)m(oment) sequence. It is
easily seen that by multiplying a pm sequence by a positive number we get
another pm sequence. Hence the set of pm sequences forms a cone $\mathcal{M}$
in the space of all sequences. To simplify notation we will denote for
simplicity integrals over $\operatorname*{supp}\mu$ by simply by $\int$.
Obviously, we have $m_{2i}\geq0,$ for $i\allowbreak=\allowbreak0,1,2\ldots$.
Using Cauchy-Schwarz inequality, we immediately have:
\[
m_{n}^{2}=\left(  \int x^{n-j+j}\mu(dx)\right)  ^{2}\leq m_{2(n-j)}m_{2j},
\]
for all $0\leq j\leq n$. Remembering that function $\left\vert x\right\vert
^{a}$ is convex for all $\alpha>1$ and after applying Jensen's inequality we
get for all $0<\alpha<\beta$%
\[
\left(  \int\left\vert x\right\vert ^{\alpha}\mu(dx)\right)  ^{\beta/\alpha
}\leq m_{0}^{\beta/\alpha-1}\int\left\vert x\right\vert ^{\beta}\mu(dx).
\]
Consequently, for the case $m_{0}\allowbreak=\allowbreak1,$ we get for all
$j\geq n$%
\[
m_{2n}^{1/(2n)}\leq m_{2j}^{1/(2j)}.
\]
Additionally, if $\operatorname*{supp}\mu\subset\lbrack0,\infty)$ then
obviously we also have seen that $m_{n}^{1/n}$ constitutes a non-decreasing sequence.

Sometimes, switching to the so-called random variables is more straightforward
and intuitive. Namely, it is known that for every positive such measure $\mu$
that $\int\mu(dx)\allowbreak=\allowbreak1$ one can define another measurable
space $(\Omega,\mathcal{B},P)$ (so-called probability space) and measurable
mapping $X:$ $\Omega\rightarrow R$ such that:
\[
P(X^{-1}(B))=\mu(B),
\]
for every measurable subset $B$ of the real line. By the way, measure $\mu$ is
called the distribution of the random variable $X$ and we obviously have:%
\[
\int_{\Omega}f(X(\omega))P(d\omega)\allowbreak=\allowbreak\int f(x)\mu(dx),
\]
for every real, integrable ($\operatorname{mod}\mu)$ function $f$ on the real
line. The distribution of any random variable is called sometimes a
probability measure on the real line. Traditionally, the integral over
$\Omega$ with respect to the probability measure $P$ is simply denoted by $E$
(-expectations). So we will interchangeably use the notation $Ef(X)\allowbreak
=\allowbreak\int f(x)\mu_{X}(dx)$, where $\mu_{X}$ denotes the distribution of
the random variable $X$.

We have important result called Hamburger moment criterion.

\begin{theorem}
A real sequence is a pm sequence iff the sequence of its Hms is positive semi-definite.
\end{theorem}

\begin{proof}
The necessity is simple. Let $X$ be a random variable with a moment sequence
$\left\{  a_{n}\right\}  $ , i.e., $a_{k}\allowbreak=\allowbreak EX^{k},$
$k\geq0.$ Then for any real sequence $\left\{  \alpha_{n}\right\}  $ then for
any $N$ we have $E\left(  \sum_{j=0}^{N}\alpha_{j}X^{j}\right)  ^{2}\geq0.$
But this quantity can be written in the following way:%
\[
\mathbf{\alpha}_{N}^{T}\left(  E\mathbf{X}_{N}\mathbf{X}_{N}^{T}\right)
\mathbf{\alpha}_{N}\geq0,
\]
where $\mathbf{\alpha}_{N}^{T}\allowbreak=\allowbreak(\alpha_{0},\ldots
,\alpha_{N}),$ $\mathbf{X}_{N}^{T}=(1,X,\ldots,X^{N})$ and $\mathbf{x}^{T}$
denotes transposition of the vector $\mathbf{x.}$ Now, this inequality means
that matrices $\left\{  E\mathbf{X}_{N}\mathbf{X}_{N}^{T}\right\}  _{N\leq0}$
are nonnegative definite, hence their major determinants must be nonnegative.

To prove the converse statement is more complicated. That is we will skip it,
since it is not the main topic of the paper.
\end{proof}

\begin{remark}
Let us recall that from the matrix theory it follows that:

1) the symmetric matrix is positive definite iff all its leading principal
minors are positive,

2) the symmetric matrix is positive semi-definite iff all its principal minors
are non-negative.
\end{remark}

\begin{remark}
Consequently, we can notice that if the sequence of Hankel transforms of a
sequence is positive, then the sequence of Hms is positive definite,
consequently, the sequence is a pm sequence.
\end{remark}

\begin{remark}
On the other hand, we have sequences having non-negative Hankel transforms
that are not positive moment sequences. One of them is the following sequence
$\{1,1,1,1,0,0,\ldots.\}$. One can see that the Hankel transform of this
sequence is $\left\{  1,0,0,1,0,\ldots\right\}  .$ The sequence
$\{1,1,1,1,0,0,\ldots.\}$ cannot be a pm sequence since $EX^{3}=1$ and
$EX^{4}\allowbreak=\allowbreak0$!
\end{remark}

Recently Berg\&Szwarc in \cite{Berg15} have made a contribution that clarifies
the case of non-negativity of the sequence of Hankel transform of a sequence
$s$ and the existence of a finitely supported positive measure generating $s$.
Namely, they proved the following result:

\begin{theorem}
[Berg\&Szwarc]Let $s\allowbreak=\allowbreak\left\{  s_{i}\right\}
_{i=0}^{\infty}$ be such a real sequence that its sequence of Hankel
transforms $\left\{  D_{j}\right\}  _{j=0}^{\infty}$ satisfy the following
condition: $\exists r>0:D_{j}>0$ for $j\leq r-1$ and $D_{j}=0$ for $j\geq r,$
then there exists a positive measure $\mu$ supported on exactly $r$ points
such that $s$ is a pm sequence and $gm(s)\allowbreak=\allowbreak\mu$.
\end{theorem}

Following the paper by Bennett \cite{Benn11}, we have yet another way of
deciding if a given sequence is a moment sequence. Namely, we have the
following criterion (compare \cite{Akhize38}).

\begin{definition}
Let $I$ be a segment of a real line. A given sequence $\left\{  m_{j}\right\}
_{j\geq0}$ is a pm sequence on $I$ iff for every polynomial $\sum_{k=0}%
^{n}c_{k}x^{k}$ that is nonnegative on $I$ we have $\sum_{k=0}^{n}c_{k}%
m_{k}\geq0.$
\end{definition}

\begin{remark}
It is well-known that if the cardinality of the support of the measure $\mu$
is infinite, then the sequence of Hankel transforms of the moment sequence of
the measure $\mu$ is strictly positive.
\end{remark}

Now, having random variables we can simply utter some rules concerning moment
sequences. Most of them were formulated in the paper of Bennett \cite{Benn11}.
Some of them had quite complicated proofs. Due to probabilistic
interpretation, we can substantially simplify these proofs. Assertions 6. 7.
and 8. of the Proposition below seem to be unknown to Bennett. Before we
formulate the results let us introduce the following notion%
\begin{equation}
h_{i,n}(\chi)\allowbreak=\allowbreak\binom{n}{i}\int_{0}^{1}\theta
^{i}(1-\theta)^{n-i}\chi(d\theta),\label{Haus}%
\end{equation}
where with $\chi$ being some probability measure on $[0,1]$. We will call
numbers $\left\{  h_{i,n}(\chi)\allowbreak\right\}  _{n\geq0,0\leq i\leq n}$
Hausdorff means.

We have:

\begin{proposition}
\label{first}Let $\left\{  a_{n}\right\}  _{n\geq0}$ and $\left\{
b_{n}\right\}  _{n\geq0}$ be two pm sequences. Then, so are the following sequences:

1. $\left\{  \alpha a_{n}+\beta b_{n}\right\}  _{n\geq0}\text{, }$for
$\alpha,\beta\geq0$.

2. $\left\{  \sum_{i=0}^{n}h_{i,n}(\chi)\alpha^{i}a_{i}\beta^{n-i}%
b_{n-i}\right\}  _{n\geq0}$, .

3. $\left\{  a_{n}b_{n}\right\}  _{n\geq0},$ $\left\{  a_{kn}\right\}
_{n\geq0}\text{, }k\in\mathbb{N}$,$\ $ $c_{n}=\left\{
\begin{array}
[c]{ccc}%
a_{2k} & if & n=2k\\
0 & if & n=2k+1
\end{array}
\right.  \ \text{,}\ $ $k=0,1,\ldots$ .

4.$\ $ If a pm sequence $\left\{  a_{n}\right\}  _{n\geq0}$ is nonnegative,
then also pm is the following sequence: $b_{n}^{\ ^{\prime}}=\left\{
\begin{array}
[c]{ccc}%
0 & if & n=2k+1\\
a_{k} & if & n=2k
\end{array}
\right.  \ \text{.}\ $

5. Suppose $\left\{  a_{n}\right\}  _{n\geq0}$ is a pm sequence, then for all
$k\geq1$ the following sequences $\left\{  a_{2k+n}\right\}  _{n\geq0}$ are
the pm sequences.

6. Suppose that for a given pm sequence $\left\{  a_{n}\right\}  _{n\geq0}$ we
have $a_{0}\allowbreak=\allowbreak1$ and $a_{2}\allowbreak=\allowbreak
a_{1}^{2},$ then $\forall n\geq0:a_{n}\allowbreak=\allowbreak a_{1}^{n}.$

7. Suppose that a sequence $\left\{  b_{n}\right\}  _{n\geq0}$ is a pm
sequence and

\qquad i) suppose further that $b_{0}\allowbreak=\allowbreak1$, and
$b_{4m+2}\allowbreak=\allowbreak b_{2m+1}^{2}$ for some $m\geq0$. Then
$\forall n\geq0:$ $b_{n}\allowbreak=\allowbreak b_{2m+1}^{n/(2m+1)}$ .

\qquad ii) suppose further that $b_{0}\allowbreak=\allowbreak1$ and
$b_{4m}\allowbreak=\allowbreak b_{2m}^{2}$ for some $m\geq1.$ Let us set
$p\allowbreak=\allowbreak(b_{1}\allowbreak+\allowbreak b_{2m}^{1/(2m)}%
)/(2b_{2m}^{1/(2m)})$ Then for all $n\geq0:b_{2n}\allowbreak=\allowbreak
b_{2m}^{n/m}$ and $b_{2n+1}\allowbreak=\allowbreak b_{2m}^{(2n+1)/(2m)}%
p-b_{2m}^{(2n+1)/(2m)}(1-p).$

8. Suppose that a sequence $\left\{  b_{n}\right\}  _{n\geq0}$ is a pm
sequence then for all $n\geq1:$
\[
b_{2n-2}b_{2n+2}\geq b_{2n}^{2}.
\]
Further, if the measure generating sequence $\left\{  b_{n}\right\}  $ has
infinite support, then the sequence $\left\{  b_{2n}\right\}  $ is log-convex.
In particular the following sequence $\left\{  b_{2n+2}/b_{2n}\right\}
_{n\geq0}$ is non-decreasing.
\end{proposition}

\begin{proof}
1. Let us define the following matrices $A_{n}\allowbreak=\allowbreak\lbrack
a_{i+j}]_{i,j=0,1,\ldots,n}$ and $B_{n}\allowbreak=\allowbreak\lbrack
b_{i+j}]_{i,j=0,1,\ldots,n}.$ Since the sequences $\left\{  \det
A_{n}\right\}  _{n\geq0}$ and $\left\{  \det B_{n}\right\}  _{n\geq0}$ are
nonnegative by assumption, then so the matrices $\left\{  A_{n}\right\}  $ and
$\left\{  B_{n}\right\}  $ are nonnegative defined hence so are their convex
combinations, that is the sequence $\left\{  \det[\alpha A_{n}+\beta
B_{n}\right\}  _{n\geq0}$ is non-negative.

2. Obviously, the sequence $\left\{  E\left(  \alpha\theta X+\beta
(1-\theta)Y\right)  ^{n}\right\}  $ where independent random variables $X$ and
$Y$ and independent also of the random variable $\theta$ are such that
$a_{n}\allowbreak=\allowbreak EX^{n}$ and $b_{n}\allowbreak=\allowbreak
EY^{n}$ for all $n\geq0$. But we have
\[
E\left(  \alpha\theta X+\beta(1-\theta)Y\right)  ^{n}=\sum_{i=0}^{n}%
h_{i,n}(\chi)\alpha^{i}a_{i}\beta^{n-i}b_{n-i}.
\]

3. If two independent random variables $X$ and $Y$ are such that
$EX^{n}\allowbreak=\allowbreak a_{n}$ and $EY^{n}\allowbreak=\allowbreak
b_{n}$, the $E(XY)^{n}\allowbreak=\allowbreak EX^{n}EY^{n}$ is also a pm
sequence. Notice, that $a_{kn}\allowbreak=\allowbreak E(X^{k})^{n}.$ Note,
that if $X$ is such a random variable that $a_{n}\allowbreak=\allowbreak
EX^{n,}$ then $(-1)^{n}a_{n}\allowbreak=\allowbreak E(-X)^{n}.$ Now, we apply
assertion $1$ with $\alpha\allowbreak=\allowbreak\beta\allowbreak
=\allowbreak1,$ $b_{n}\allowbreak=\allowbreak(-1)^{n}a_{n}$ .

4. First notice that if a pm sequence $\left\{  a_{n}\right\}  _{n\geq0}$ is
nonnegative, then the supporting measure must be concentrated on $[0,\infty)
$. We apply assertion 1. with $\alpha\allowbreak=\allowbreak\beta
\allowbreak=\allowbreak1,$ $a_{n}\allowbreak=\allowbreak E(\sqrt{\left\vert
X\right\vert })^{n},$ $b_{n}\allowbreak=\allowbreak(-1)^{n}a_{n} $ for some
random variable $X$. Now the sequence $\left\{  a_{n}+b_{n}\right\}  $
contains zeros for odd indices and $2E\left\vert X\right\vert ^{k}$ for
$n\allowbreak=\allowbreak2k$. But we can always divide elements of a pm
sequence by a positive number. See also \cite{Chih79}, p. 40.

5. Let us notice that if for some $k$ we have $a_{2k}\allowbreak
=\allowbreak0,$ then the supporting measure of $X$ must be equal to zero,
i.e., $\mu(B)\allowbreak=\allowbreak0$ for all Borel sets $B.$ Consequently,
we would have $a_{0}\allowbreak=\allowbreak0$ and the statement assertion
would be true. So let us assume that $a_{0}>0$ and consequently that
$a_{2k}>0$ for all $k\geq0.$ Let $X$ be such a random variable that
$EX^{n}\allowbreak=\allowbreak a_{n}/a_{0}.$ Let $\mu(.)$ denote the
distribution of $X$ and let us denote by $\nu_{k}$ the probability measure
$x^{2k}d\mu(x)/a_{2k}.$ Let us denote $Y_{k}$ the random variable that has
distribution $\nu_{k}.$ Then $EY_{k}^{n}\allowbreak=\allowbreak\frac{1}%
{a_{2k}}\int y^{n}y^{2k}d\mu(y)\allowbreak=\allowbreak\frac{1}{a_{2k}%
}EX^{2k+n}\allowbreak=\allowbreak a_{2k+n}/a_{2k}.$ hence $\left\{
a_{2k+n}/a_{2k}\right\}  _{n\geq0}$ is a pm sequence and by assertion 1. also
$\left\{  a_{2k+n}\right\}  _{ns>0}$ is a pm sequence.

6. Let $X$ denote a random variable whose moments are $\alpha_{n},$ that is
$EX^{n}\allowbreak=\allowbreak\alpha_{n}.$ We have $\operatorname*{var}%
(X)\allowbreak=\allowbreak EX^{2}\allowbreak-\allowbreak(EX)^{2}%
\allowbreak=\allowbreak\rho^{2}-\rho^{2}\allowbreak=\allowbreak0.$ But this
equality means that the distribution of $X$ is a one-point distribution, i.e.
$P(X=\rho)\allowbreak=\allowbreak1.$

7. For the proof see Lemma 2.2 of \cite{Szabl22}.

8. Since by the fact that $\left\{  b_{n}\right\}  _{n\geq0}$ is a pm
sequence, then all central minors of the matrix $\left[  b_{i+j}\right]
_{i,j\geq0}$ must be nonnegative and moreover, in particular, we must have%
\begin{equation}
b_{2n-2}b_{2n+2}\geq b_{2n}^{2}, \label{b2n}%
\end{equation}
for all $n\geq1.$ Now, if the measure that produces the sequence $\left\{
b_{n}\right\}  $ has infinite support, then $\forall n\geq0:b_{2n}>0.$ Hence,
\[
\log b_{2(n-1)}+\log b_{2(n+1)}\geq2\log b_{2n},
\]
it proves log-convexity. On the other hand, we can easily deduce from
(\ref{b2n}) that $b_{2n+2}/b_{2n}\geq$ $b_{2n}/b_{2n-2}$ since $b_{2n}>0.$
\end{proof}

\begin{remark}
Let us consider Haussdorf means $\left\{  h_{i,n}(\chi)\right\}  $. We have
$h_{i,n}(\chi)\geq0$ and moreover $\forall n\geq0:\sum_{i=0}^{n}h_{i,n}%
(\chi)\allowbreak=\allowbreak1.$ Hence, for any number sequence $\left\{
\gamma_{n}\right\}  _{n\geq0}$ the sums
\[
\sum_{i=0}^{n}h_{i,n}(\chi)\gamma_{i},
\]
are the kind of averages.

Let us give some examples.

i) Take measure $\chi$ to be a one-point probability measure concentrated at
$t\in(0,1).$ Then we have $h_{i,n}(\chi)\allowbreak=\allowbreak\binom{n}%
{i}t^{i}(1-t)^{n-i}.$ Let is take $b_{n}\allowbreak=\allowbreak c,$ and
$\alpha\allowbreak=\allowbreak\beta\allowbreak=\allowbreak1,$ then from
assertion 2. it follows that $\left\{  c\sum_{j=0}^{n}\binom{n}{j}t^{j}%
a_{j}(1-t)^{n-j}\right\}  $ is a pm sequence provided $\left\{  a_{j}\right\}
$ is. Further taking $\alpha\allowbreak=\allowbreak1/2,$ $t\allowbreak
=\allowbreak1/2,$ we see that the so-called binomial transform of the sequence
$\left\{  a_{j}\right\}  ,$ , i.e., $\left\{  \sum_{j=0}^{n}\binom{n}{j}%
a_{j}\right\}  $ is also a pm sequence. If we take $\alpha\allowbreak
=\allowbreak-\beta\allowbreak=\allowbreak-1\allowbreak=\allowbreak-2t,$ then
also the following sequence $\left\{  \sum_{j=0}^{n}(-1)^{j}\binom{n}{j}%
a_{j}\right\}  $ which is called inverse binomial transform is a pm sequence.

ii) Take $\chi(dx)\allowbreak=\allowbreak\beta(1-x)^{\beta-1},$ for some
$\beta\geq0.$ Then
\[
h_{i,n}(\chi)\allowbreak=\allowbreak\beta\binom{n}{i}\frac{\Gamma
(i+1)\Gamma(n-i+\beta)}{\Gamma(n+\beta+1)}=\frac{\beta n!\Gamma(n-i+\beta
)}{\Gamma(n+\beta)(n-i)!}.
\]
In particular, if $\beta\allowbreak=\allowbreak1$ we have $h_{i.n}%
(\chi)\allowbreak=\allowbreak\frac{1}{n+1}.$ Hence, following assertion 2.
Proposition \ref{first} we get%
\begin{equation}
\frac{1}{n+1}\sum_{i=0}^{n}a_{i}b_{n-i} \label{av_con}%
\end{equation}
is a pm sequence provided $\left\{  a_{i}\right\}  $ and $\left\{
b_{n}\right\}  $ are.

Given some two sequences $\left\{  a_{i}\right\}  _{i\geq0}$ and $\left\{
b_{i}\right\}  _{i\geq0}$ the following sequence \newline$\left\{  \sum
_{j=0}^{i}a_{j}b_{i-j}\right\}  _{i\geq0}$ is called the \textbf{convolution}
of the sequences $\left\{  a_{i}\right\}  _{i\geq0}$ and $\left\{
b_{i}\right\}  _{i\geq0}$. Hence we see that sequence of arithmetic means of
the \emph{convolution} of two pm sequences is also a pm sequence.

More particular cases and their applications can be found in \cite{Benn11}
(formulae (31)--(34)).
\end{remark}

As a corollary from the interesting results of Layman (see \cite{Layman2001})
we have the following result.

\begin{lemma}
The following two pm sequences $\left\{  a_{n}\right\}  _{n\geq0}$ and
$\left\{  \sum_{j=0}^{n}\binom{n}{j}(-a_{1})^{j}a_{n-j}\right\}  _{n\geq0}$
have the same Hankel transforms provided $a_{0}\allowbreak=\allowbreak1$.
\end{lemma}

\begin{proof}
Let us denote elements of the second sequence by $\hat{a}_{n}.$ First if
$a_{1}\allowbreak=\allowbreak0$ then the assertion is true. So let us assume
that $a_{1}\neq0.$ Secondly notice $\det\left[  a_{i+j}\right]  _{i,j=0,1}%
\allowbreak=\allowbreak a_{0}a_{2}-a_{1}^{2},$ while $\det\left[  \hat
{a}_{i+j}\right]  _{i,j=0,1}\allowbreak=\allowbreak a_{0}(a_{2}-a_{1}^{2})$.
Hence, to have equality of the two determinants we have to assume
$a_{0}\allowbreak=\allowbreak1$. Let us notice that we have
\begin{equation}
\sum_{j=0}^{n}\binom{n}{j}(-a_{1})^{j}a_{n-j}\allowbreak=\allowbreak
(-a_{1})^{n}\sum_{j=0}^{n}\binom{n}{j}a_{n-j}/(-a_{1})^{n-j}. \label{pom}%
\end{equation}
Now notice that $\sum_{j=0}^{n}\binom{n}{j}a_{n-j}/(-a_{1})^{n-j}$ is the
binomial transform of the sequence $\left\{  a_{j}/(-a_{1})^{j}\right\}
_{j\geq0}$. By the Layman's Theorem 1, these two sequences have the same
Hankel transforms. Now, the sequence $\left\{  a_{j}/(-a_{1})^{j}\right\}
_{j\geq0}$ has the Hankel transform equal to $\left\{  \det[a_{i+j}]_{0\leq
i,j\leq n}/(-a_{1})^{2n}\right\}  .$ Hence the sequence $\left\{  \sum
_{j=0}^{n}\binom{n}{j}a_{n-j}/(-a_{1})^{n-j}\right\}  $ has the same Hankel
transform. Now by (\ref{pom}) we deduce that the sequence\newline$\left\{
\sum_{j=0}^{n}\binom{n}{j}(-a_{1})^{j}a_{n-j}\right\}  _{n\geq0}$ has the
Hankel transform equal to \newline$\left\{  (-a_{1})^{2n}\det[a_{i+j}]_{0\leq
i,j\leq n}/(-a_{1})^{2n}\right\}  \allowbreak=\allowbreak\left\{  \det
[a_{i+j}]_{0\leq i,j\leq n}\right\}  $.
\end{proof}

\begin{remark}
From the above-mentioned lemma it follows that the random variables $X$ and
$X-EX$ have the same Hankel transforms. In other words, the sequence of
(ordinary) moments of a random variable and the sequence of its central
moments have the same Hankel transforms.
\end{remark}

As before majority of the assertions of the Proposition below are known. We
provide their simple, probabilistic proofs.

\begin{proposition}
\label{second}The following sequences are pm sequences:

1) $\forall a\in\mathbb{R}:$ $\left\{  a^{n}\right\}  $,

2) $\left\{  n!\right\}  _{n\geq0},$

3) $\left\{
\begin{array}
[c]{ccc}%
1 & if & n=0\\
0 & if & n\text{ is odd}\\
(2k-1)!! & if & n=2k
\end{array}
\right.  ,\allowbreak k=1,2,\ldots\ \text{,}\ $

4) Catalan numbers i.e.\newline$\left\{  \binom{2n}{n}/(n+1)\right\}
_{n\geq0}\ \text{,}\ $ $\ $

5) $\forall k>-1:\{1/(n+1)^{k+1}\}_{n\geq0},$~

6) $\forall\alpha\geq0:\left\{  \left(  \alpha\right)  ^{(n)}\right\}
_{n\geq0}$, where $\left(  \alpha\right)  ^{(n)}\allowbreak=\allowbreak
\alpha(\alpha+1)\ldots(\alpha+n-1)$ is the so called raising factorial of
$\alpha$,

7) $\forall\alpha,\beta>0:\left\{  \frac{\left(  \alpha\right)  ^{(n)}%
}{\left(  \alpha+\beta\right)  ^{(n)}}\right\}  _{n\geq0}$, consequently
$\forall\alpha>0:\left\{  \left(  \alpha\right)  ^{(n)}/n!\right\}  _{n\geq0}$,

8) $\left\{  F_{n+1}\right\}  _{n\geq0}\ $,$\ $ $\left\{  F_{n+3}\right\}
_{n\geq0},$ $\left\{  F_{n+5}\right\}  _{n\geq0},\ldots$, $\left\{
F_{2n+2}\right\}  _{n\geq0}$, $\left\{  F_{n+1}/(n+1)\right\}  _{n\geq
0}\ \text{,}\ $ $\left\{  F_{2n+2}/(n+1)\right\}  _{n\geq0}\ \text{,}\ $
$\left\{  (F_{n+2}-1)/(n+1)\right\}  _{n\geq0}$ , $\left\{  (F_{2n+1}%
-1)/(n+1)\right\}  _{n\geq1},$ for any natural $k$ $\left\{  1/F_{2k+n}%
\right\}  _{n\geq0}$, where $F_{n}$ denotes $n-$th Fibonacci number,.

9) $\forall\lambda\geq0:\left\{  \sum_{j=0}^{n}\lambda^{j}%
\genfrac{\{}{\}}{0pt}{}{n}{j}%
\right\}  _{n\geq0}$ where $%
\genfrac{\{}{\}}{0pt}{}{n}{j}%
$ denotes Stirling number of the second kind.

10) $\left\{  B_{n}\right\}  _{n\geq0},$ $\left\{  B_{n+1}\right\}  _{n\geq0}$
where $B_{n}$ is Bell number.
\end{proposition}

\begin{proof}
1) $\left\{  a^{n}\right\}  _{n\geq0}$ is the moment sequence of a constant
random variable. 2) $n!\allowbreak=\allowbreak\int_{0}^{\infty}x^{n}%
\exp(-x)dx,$ 3) $\left\{
\begin{array}
[c]{ccc}%
1 & if & n=0\\
0 & if & n\text{ is odd}\\
(2k-1)!! & if & n=2k
\end{array}
\right.  \allowbreak=\allowbreak\int_{-\infty}^{\infty}x^{n}\frac{\exp
(-x^{2}/2)}{\sqrt{2\pi}}dx,$ 4) $\binom{2n}{n}/(n+1)\allowbreak=\allowbreak
\frac{1}{2\pi}\int_{0}^{4}x^{n}\sqrt{(4-x)/x}dx,$ 5) $1/(n+1)^{k+1}%
\allowbreak=\allowbreak\int_{0}^{1}x^{n}\frac{(-\log(x))^{k}}{\Gamma(k+1)}dx$,
6) $\left(  \alpha\right)  ^{(n)}\allowbreak=\allowbreak\frac{1}{\Gamma
(\alpha)}\int_{0}^{\infty}x^{n}x^{\alpha-1}\exp(-x)dx$, 7) $\frac{\left(
\alpha\right)  ^{(n)}}{\left(  \alpha+\beta\right)  ^{(n)}}\allowbreak
=\allowbreak\frac{1}{B(\alpha,\beta)}\int_{0}^{1}x^{n}x^{\alpha-1}%
(1-x)^{\beta-1}dx$, where $B(\alpha,\beta)$ is the value of beta function at
$\alpha$ and $\beta,$ Now taking $\beta\allowbreak=\allowbreak k\allowbreak
-\allowbreak\alpha$ we deduce that $\left\{  \frac{(\alpha)^{(n)}}%
{(k+n-1)!}\right\}  $. 8) Fibonacci numbers are defined as the solution of the
following difference equation $F_{n+2}\allowbreak=\allowbreak F_{n+1}+F_{n}.$
with $F_{0}\allowbreak=\allowbreak0,$ $F_{1}\allowbreak=\allowbreak1.$ Hence
\[
F_{n}\allowbreak=\allowbreak\frac{1}{\sqrt{5}}\left(  \frac{1+\sqrt{5}}%
{2}\right)  ^{n}-\frac{1}{\sqrt{5}}\left(  \frac{1-\sqrt{5}}{2}\right)  ^{n}.
\]
Now
\[
F_{n+2k+1}\allowbreak=\allowbreak\frac{1}{\sqrt{5}}\left(  \frac{1+\sqrt{5}%
}{2}\right)  ^{2k+1}\left(  \frac{1+\sqrt{5}}{2}\right)  ^{n}-\frac{1}%
{\sqrt{5}}\left(  \frac{1-\sqrt{5}}{2}\right)  ^{2k+1}\left(  \frac{1-\sqrt
{5}}{2}\right)  ^{n}.
\]
Now notice that $\forall k\geq0$ both $\frac{1}{\sqrt{5}}\left(  \frac
{1+\sqrt{5}}{2}\right)  ^{2k+1}$ and $-\frac{1}{\sqrt{5}}\left(  \frac
{1-\sqrt{5}}{2}\right)  ^{2k+1}$ are positive. Now we apply assertion 1 of
Proposition \ref{first} We have also
\[
F_{n+3}=\frac{1}{\sqrt{5}}\left(  \frac{1+\sqrt{5}}{2}\right)  ^{n}-\frac
{1}{\sqrt{5}}\left(  \frac{1-\sqrt{5}}{2}\right)  ^{3}\left(  \frac{1-\sqrt
{5}}{2}\right)  ^{n}.
\]
But $-\frac{1}{\sqrt{5}}\left(  \frac{1-\sqrt{5}}{2}\right)  ^{3}%
\allowbreak=\allowbreak1\allowbreak-\allowbreak\frac{2}{\sqrt{5}}>0.$
Similarly we can consider $F_{n+5}$ and generally $F_{n+2k+1}.$ They
constitute pm sequence (no probabilistic in general) since $-\frac{1}{\sqrt
{5}}\left(  \frac{1-\sqrt{5}}{2}\right)  ^{2k+1}>0$ for $k\geq0.$ Applying
third assertion of the above-mentioned Proposition \ref{first}, we deduce that
$\left\{  F_{2n+2}\right\}  _{n\geq0}.$ The fact that $\left\{  F_{n+1}%
/(n+1)\right\}  _{n\geq0}$ and $\left\{  F_{2n+2}/(n+1)\right\}  _{n\geq0}$
are pm sequences follows third assertion of the above-mentioned Proposition
\ref{first}, where $a_{n}\allowbreak=\allowbreak F_{n+1}$ or $a_{n}%
\allowbreak=\allowbreak F_{2n+2}$ and $b_{n}\allowbreak=\allowbreak1/(n+1).$
Similarly, using the well known (see e.g. \cite{Grimaldi12}) property of
Fibonacci numbers $F_{1}\allowbreak+\allowbreak\ldots\allowbreak+\allowbreak
F_{n}\allowbreak=\allowbreak F_{n+2}\allowbreak-\allowbreak1$ and then
(\ref{av_con}) with $a_{n}\allowbreak=\allowbreak F_{n+1}$ and $b_{n}%
\allowbreak=\allowbreak1$ we see that $\left\{  (F_{n+2}-1)/(n+1)\right\}
_{n\geq0}$ is a pm sequence. Now we apply assertion 3 of Proposition
\ref{first} and use the fact that $2n+1.$ The last case i.e. proof that the
sequence $\left\{  1/F_{2k+n}\right\}  _{n\geq0} $ is a pm sequence for every
natural $k$ is treated and proved in the excellent paper by Berg
\cite{Berg11}. 9) We start with the observation that if $X$ has the so-called
Poisson distribution with parameter $\lambda\geq0$ , i.e.,%
\[
P(X=k)=\lambda^{k}\exp(-\lambda)/k!,
\]
then
\[
EX^{n}\allowbreak=\allowbreak\sum_{j=0}^{n}%
\genfrac{\{}{\}}{0pt}{}{n}{j}%
\lambda^{j}.
\]
10) Now, we use the fact that the so-called Bell numbers $B_{n}$ are defined
as
\[
B_{n}\allowbreak=\allowbreak\sum_{j=0}^{n}%
\genfrac{\{}{\}}{0pt}{}{n}{j}%
.
\]
We also use the well-known fact that
\[
B_{n+1}\allowbreak=\allowbreak\sum_{j=0}^{n}\binom{n}{j}B_{j}%
\]
and apply assertion 2) of the Proposition \ref{first} with $h_{i,n}%
(\chi)\allowbreak=\allowbreak\binom{n}{i}.$
\end{proof}

\begin{lemma}
If $\left\{  a_{n}\right\}  _{n\geq0}$ is a pm sequence, then the following
sequence of polynomials indexed by $n\geq0$
\[
\sum_{j=0}^{2n}a_{j}x^{j}/j!
\]
assumes only nonnegative values for $x\in\mathbb{R}$.
\end{lemma}

\begin{proof}
If $x\allowbreak=\allowbreak0,$ then the assertion is true. Hence assume that
$x\neq0$. We consider sequence of Haussdorf means with $\left\{
a_{n}\right\}  ,$ $\left\{  b_{n}\allowbreak=\allowbreak x^{-n}n!\right\}  $,
$\alpha\allowbreak=\allowbreak1/d,$ $\beta\allowbreak=\allowbreak1/(1-d),$
$\chi$ being the one point measure concentrated at point $d\in(0,1),$
Consequently we have
\[
h_{i,n}(\chi)=\binom{n}{i}d^{i}\left(  1-d\right)  ^{n-i}.
\]
Now following Proposition \ref{first}(2) we deduce that the following sequence
is a pm sequence:%
\[
\sum_{j=0}^{n}\binom{n}{i}a_{j}x^{-n+j}(n-j)!\allowbreak=\allowbreak\frac
{n!}{x^{n}}\sum_{j=0}^{n}a_{j}x^{j}/j!.
\]
Now since all elements of a pm sequence with even indexes are positive and
since $1/x^{2n}>0$ for all $x\neq0$ and $\left(  2n\right)  !$ is positive we
get our assertion.
\end{proof}

\section{Moment sequences and linear difference equations\label{MD}}

First, let us fix the terminology. The best one seems to be taken from the
systems theory. So let us consider the following difference equation with
constant parameters:%
\begin{equation}
\sum_{j=0}^{m}d_{j}r_{n+j}=c_{n}, \label{diff}%
\end{equation}
with complex, in general parameters$\left\{  c_{n}\right\}  $, $d_{j},$
$j\allowbreak=\allowbreak0,\ldots,m,$ and $d_{m}\allowbreak\neq\allowbreak0$
and with also complex in general so-called initial conditions: $r_{0}%
=p_{0},\ldots,r_{m-1}=p_{m-1}.$

If $\forall n\geq0:c_{n}=0,$ then the equation is called homogeneous otherwise
it is called non-homogeneous.

Recall that we dealt with this type of difference equations, above, when
discussing fR-sequences.

For the sake of the completeness of the paper, let us recall the basic
properties of the linear difference equations with constant parameters.

Given the so-called system parameters: $\left\{  d_{j}\right\}  _{j=0}^{m}$,
(sometimes one says about input-output system parameters) the input sequence
$\left\{  c_{n}\right\}  _{n\geq0}$ and initial conditions, the task is to
find the so-called output sequence $\left\{  r_{n}\right\}  _{n\geq0}$. The
equation (\ref{diff}) with the set of initial values is called initial value
problem. It is known that every the initial value problem has always a
solution, i.e. the sequence $\left\{  r_{n}\right\}  $ is determined uniquely.
In other words, the set of parameters, the input sequence and the set of
initial conditions determine the output sequence uniquely.

Moreover, the parameter $m$ is called the order of the difference equation and
the following algebraic equation:%
\begin{equation}
\sum_{j=0}^{m}d_{j}x^{j}=0, \label{char}%
\end{equation}
is called a characteristic equation of the difference equation (\ref{diff}).
The roots of the characteristic equation are very important since they enable
the construction of the solution of the homogeneous equation. The formula for
this solution is very simple, if the roots are all different. The case of the
different roots of the characteristic equation is also important for the
problem of checking if the solution is a pm sequence. Hence we will present
this solution for the case of different roots only. For the case of multiple
roots we direct the reader to the monograph \cite{Cull05}.

Notice that when roots of the characteristic equation (\ref{char}) are all
different and equal to $b_{j},j\allowbreak=\allowbreak1,\ldots,m$ we have the
alternative form of the characteristic equation namely
\[
\prod_{j=1}^{m}(x-b_{i})=0.
\]
On the other hand we have expansion
\[
\prod_{j=1}^{m}(x-b_{i})=\sum_{j=0}^{m}x^{m-j}(-1)^{j}S_{j}(\mathbf{b),}%
\]
where $S_{j}(\mathbf{b)}$ denotes a $j$-th simple symmetric function of the
numbers $b_{1},\ldots,b_{m}$ with $S_{0}(\mathbf{b)\allowbreak=\allowbreak1.}$
Hence, we have relationship between parameters $\left\{  d_{j}\right\}
_{j=0}^{m}$ and numbers $\left\{  b_{j}\right\}  _{j=1}^{m}.$

We start with the homogeneous linear difference equations with constant
coefficients. Obviously the following sequence of numbers:
\[
r_{n}=\sum_{k=1}^{m}\alpha_{k}\left(  b_{k}\right)  ^{n},
\]
for different real $b_{k},$ and non-negative $\alpha_{k},$ $k\allowbreak
=\allowbreak1,\ldots,m$ is a pm sequence. By the way, the support of the
measure that generates this sequence has finite cardinality, equal to $m$ and
consists of points of the set $\left\{  b_{1},\ldots,b_{m}\right\}  $. On the
other hand, it is known that this sequence, i.e., $\left\{  r_{n}\right\}
_{n\leq0}$ is a solution of the following difference equation%
\begin{equation}
\sum_{j=0}^{m}(-1)^{j}S_{j}(\mathbf{b})r_{n+m-j}=0 \label{hom}%
\end{equation}
with initial conditions $r_{0}\allowbreak=\allowbreak p_{0},\ldots
,r_{m-1}\allowbreak=\allowbreak p_{m-1}$ selected is such a way that
coefficients $\left\{  \alpha_{k}\right\}  _{k=1,\ldots,m}$ are non-negative.

It is known that there exists a direct one-to-one relationship between numbers
$\left\{  p_{0},\ldots,p_{m-1}\right\}  $ and the numbers $\left\{
a_{k}\right\}  _{k=1,\ldots,m}$ with an obvious relationship:
\[
\sum_{j=1}^{m}a_{j}b_{j}^{k}=p_{k},
\]
for $k\allowbreak=\allowbreak0,\ldots m-1$. $.$

The fact that in the case of pm sequences only different roots are concerned
can be seen, when studying the following example. Let's take $m\allowbreak
=\allowbreak2$ and $b_{1}\allowbreak=\allowbreak b_{2}\allowbreak=\allowbreak
a$ i.e. we consider the following equation
\[
r_{n+2}-2ar_{n+1}+a^{2}r_{n}=0,
\]
with $r_{0}\allowbreak=\allowbreak1$ and $r_{1}\allowbreak=\allowbreak r1$. It
is not difficult to check that the following sequence
\[
r_{n}=a^{n-1}(nr_{1}-a(n-1)),
\]
for $n\geq0$ satisfies the equation. Now, as it is also easy to check, the
Hankel transform of this sequence is $1,-(r1-a)^{2},0,\ldots$ . Hence, unless
$a\allowbreak=\allowbreak r1,$ for no real $a$ and $r1$, $\left\{
r_{n}\right\}  $ is a pm sequence. If $r1\allowbreak=\allowbreak a$ we have
the obviously trivial sequence $r_{n}\allowbreak=\allowbreak a^{n}.$

One has to underline that the case of two complex but conjugate roots of the
characteristic equation also leads to the non-pm case. More precisely, if one
considers, for example, the difference equation:
\[
r_{n+2}+a^{2}r_{n}=0,
\]
with $r_{0}\allowbreak=\allowbreak1,$ $r_{1}\allowbreak=\allowbreak r1.$ Then
$\left\{  r_{n}\right\}  $ is not a pm sequence since we have $r_{2}%
\allowbreak=\allowbreak-a^{2}<0.$

The situation becomes more complicated and generally different when one
considers nonhomogeneous linear equations with constant coefficients. That is
when one considers the following difference equation:%
\begin{equation}
\sum_{j=0}^{m}(-1)^{j}S_{j}(\mathbf{b})r_{n+m-j}=c_{n}, \label{nhom}%
\end{equation}

with $\left\{  c_{n}\right\}  $ being a pm sequence. First, let us consider
the case when $\left\{  c_{n}\right\}  $ is a pm sequence generated by a
discrete distribution with finite support.

As an example, let us take $m\allowbreak=\allowbreak1$ and $c_{n}%
\allowbreak=\allowbreak d^{n}$.

That is, let us consider such an equation%
\[
r_{n+1}-ar_{n}=d^{n},
\]
with an initial condition $r_{0}=p_{0}.$ As it is commonly known the solution
of such an equation is given by the formula:
\[
r_{n}=\frac{a^{n}-d^{n}}{a-d}+p_{0}a^{n},
\]
for $n\geq0.$ Finding Hankel transform of this sequence we get: $p_{0}%
,p_{0}(d-a)-1,0,\ldots$ . Hence $\left\{  r_{n}\right\}  $ is a pm sequence
iff $p_{0}\geq0$ and $p_{0}(d-a)-1\geq0$. In particular, if $d>a$ then
$p_{0}\geq\frac{1}{d-a}.$

Let us note, that if we consider a second-order nonhomogeneous equation of the
following form:%
\[
r_{n+2}+a^{2}r_{n}=b^{n},
\]
with $r_{0}\allowbreak=\allowbreak r0$ and $r_{1}\allowbreak=\allowbreak r1.$
One can easily split solving this equation into two parts. First to consider
the case of even $n$ and then the odd $n$ case. Anyway, solving this equation
doesn't cause any difficulty. With the help of Mathematica, one obtains the
following sequence of Hankel transform of the sequence $\left\{
r_{n}\right\}  .$ Namely, we get:%
\[
r0,r0-a^{2}r0-r1^{2},-b^{2}\left(  r1^{2}+a^{2}r0^{2}\right)  +2r1b-1+2a^{2}%
r0-a^{2}(a^{2}r0^{2}+r1^{2}),0,\ldots.
\]
Now, it is enough to notice that the polynomial is $-b^{2}\left(  r1^{2}%
+a^{2}r0^{2}\right)  +2r1b-1+2a^{2}r0-a^{2}(a^{2}r0^{2}+r1^{2})$ negative for
all $b.$ Consequently sequence $\left\{  r_{n}\right\}  $ cannot be a pm
sequence for any $b$ different from zero. We will see in a moment that this is
not the case when one considers sequence $\left\{  c_{n}\right\}  $ being a
moment sequence of absolutely continuous measure.

To avoid unnecessary complications we will consider signed measures $d\mu$
defined on the real line that satisfy the so-called Cramer's condition, that
is that there exists $\delta>0$ such that
\begin{equation}
\int\exp(\delta\left\vert x\right\vert )d\left\vert \mu\right\vert (x)<\infty.
\label{IM}%
\end{equation}
It is known, that if a measure satisfies this condition, then it can be
identified by its moments. Let us call the set of such measures $Cra.$

Now let us consider a positive measure $dA\in Cra$ and a polynomial $Q(x)$
both such that
\begin{equation}
dB(x)\allowbreak=\allowbreak\frac{1}{Q(x)}dA(x) \label{BAP}%
\end{equation}
Note that following Proposition 1 of \cite{Szabl21} if only $dA\in Cra$ and
\begin{equation}
\int\frac{1}{Q(x)^{2}}dA(x)<\infty\text{ and }1/Q(x)\geq0 \label{condP}%
\end{equation}
on the $\operatorname*{supp}A,$ then $dB\in Cra$ and it is a positive measure.

\begin{theorem}
\label{main}Let sequences, respectively $\left\{  a_{n}\right\}  _{n\geq0}$
and $\left\{  b_{n}\right\}  _{n\geq0}$ be pm sequences generated by the
measures $dA$ and $dB\in Cra,$ related to one another by (\ref{BAP}) with
polynomial $Q$ defined by%
\[
Q(x)=\sum_{j=0}^{m}c_{j}x^{j},
\]
with $c_{m}\allowbreak\neq\allowbreak0$. Then, the sequences are related to
one another by the following difference equation:%
\begin{equation}
\sum_{j=0}^{m}c_{j}b_{n+j}=a_{n}, \label{DE}%
\end{equation}
with initial conditions
\begin{equation}
b_{k}\allowbreak=\allowbreak\int\frac{x^{k}}{Q(x)}dA(x), \label{init}%
\end{equation}
$k\allowbreak=\allowbreak0,\ldots,m-1.$ iff $Q\left(  x\right)  \geq0$ on
$\operatorname*{supp}A.$
\end{theorem}

\begin{proof}
We start with an obvious observation that for $dA$ and $dB$ to positive
measures, we have to have $Q\left(  x\right)  \geq0$ on the
$\operatorname*{supp}A.$ Further we have equality: $\forall n\geq0$:
$a_{n}\allowbreak=\allowbreak\int x^{n}dA(x).$ Now, we have $a_{n}%
\allowbreak=\allowbreak\int Q(x)x^{n}dB(x)\allowbreak=\allowbreak\sum
_{j=0}^{m}c_{j}b_{j+n}$. To get unique solution of this equation, we need to
set $m$ initial conditions that can be found by calculating $m$ numbers
defined by (\ref{init}). Conversely, assuming that both sequences $\left\{
a_{n}\right\}  _{n\geq0}$ and $\left\{  b_{n}\right\}  _{n\geq0}$ are pm
sequences of measures respectively $dA,$ $dB\in Cra$ we deduce that for
$n\geq0$
\[
\int x^{n}dA\left(  x\right)  =\int\left(  \sum_{j=0}^{m}c_{j}x^{j}\right)
x_{n}dB(x),
\]
Now, since $dA,$ $dB\in Cra$ we have uniqueness of the measure that generates
given sequence of moments. Hence, we deduce that we have%
\[
dA(x)=Q(x)dB(x).
\]
In particular that $Q\left(  x\right)  \geq0$ on the $\operatorname*{supp}A,$
since supports of the two measures are the same.
\end{proof}

\begin{remark}
Notice that knowing, say the sequence $\left\{  a_{n}\right\}  _{n\geq0}$ and
the polynomial $Q(x)$ satisfying (\ref{condP}), we can expand $1/Q(x)$ in an
infinite series%
\[
1/Q(x)=\sum_{j\geq0}d_{j}x^{j},
\]
and then we have
\[
b_{k}=\sum_{j\geq0}d_{j}a_{j+k},
\]
for $k\allowbreak=\allowbreak0,\ldots,m-1.$
\end{remark}

\begin{remark}
In other words the difference equation (\ref{DE}) with an input of the form of
a pm sequence $\left\{  a_{n}\right\}  _{n\geq0}$ generated by a positive
measure $dA\in Cra$ has the solution that is also a pm sequence generated by
the measure $dB$ also belonging to $Cra$ iff its characteristic polynomial
$Q(x)\geq0$ on $\operatorname*{supp}A$. Notice also that from the fact that
$dB\in Cra$ it follows that zeros of $Q\left(  x\right)  $ cannot occur at
points where $dA$ has atoms, similarly zeros of odd multiplicity cannot occur
at points of increase of the measure $dA.$ \footnote{Point of increase of a
positive measure $\mu$ is such a point $x$ for which $\mu\left(  B\right)  >0$
whenever $x\in B$ for an open set $B.$}. In particular, it has roots of odd
multiplicity outside the support of the measure $dA.$
\end{remark}

\section{Remarks and examples}

Let us return to the examples that were analyzed above.

\begin{example}
We start with the example with two complex conjugate roots of the
characteristic equation Now, let us consider the following equation:%
\[
r_{n+2}+a^{2}r_{n}=b^{n},
\]
with $r_{0}\allowbreak=\allowbreak\frac{1}{a^{2}+b^{2}},$ $r_{1}%
\allowbreak=\allowbreak r1.$ It turns out that the Hankel transform of $r_{n}$
is equal to $\frac{1}{a^{2}+b^{2}},$ $\frac{(b-r1(a^{2}+b^{2}))(b+r1(a^{2}%
+b^{2}))}{(a^{2}+b^{2})^{2}},$ $\frac{-(b-r1(a^{2}+b^{2}))^{2}}{a^{2}+b^{2}%
},0,\ldots,$ . Hence, unless $r1\allowbreak=\allowbreak\frac{b}{a^{2}+b^{2}}$
the solution of the above-mentioned equation cannot be a pm sequence. Thus,
the fact that the sequence $\left\{  r_{n}\right\}  $ is a pm sequence heavily
depends on the initial conditions.
\end{example}

Let us consider one more example.

\begin{example}
Namely, let us consider similar equation excited, this time by a sequence
$\left\{  \frac{1}{n+1}\right\}  _{n\geq0}.$ That is, consider the following
difference equation:%
\[
r_{n+2}+r_{n}=\frac{1}{n+1}.
\]
Now let us recall, that we have:
\[
\int_{0}^{1}x^{n}dx=\frac{1}{n+1}.
\]
So the measure $dA$ has the density equal to $1$ for $x\in\lbrack0,1]$ and
zero otherwise. According to the theorem above, the initial conditions should
be:%
\begin{align*}
r_{0}\allowbreak &  =\allowbreak\int_{0}^{1}\frac{1}{x^{2}+1}dx=\frac{\pi}%
{4},\\
r_{1}  &  =\int_{0}^{1}\frac{x}{x^{2}+1}dx=\frac{1}{2}\log2.
\end{align*}
We see that given these initial conditions we see that
\[
dB(x)=\frac{1}{1+x^{2}}\mathbf{1}_{[0,1]}(x)dx,
\]
and consequently
\[
r_{n}=\int_{0}^{1}\frac{x^{n}}{1+x^{2}}dx.
\]
Now we can get a few first elements of the Hankel transform of the sequence :
$0.785398$, $0.0484346$, $0.000201726$, $5.41176\times10^{-8}$, $9.22425\times
10^{-13}$, $9.9286\times10^{-19},\ldots$ . Now, if we only slightly change the
value of, say, $r_{1}$ by considering, say $r_{1}\allowbreak=\allowbreak
,01\allowbreak+\allowbreak.5\times\log2,$ then, we get the following sequence
of Hankel transforms: $0.785398$, $0.0414032$, $0.00082643$, $-0.0000171229$,
$-1.10472\times10^{-7}$, $-5.22732\times10^{-11}$, $-1.57765\times10^{-15}$.
\end{example}

\begin{remark}
This example suggests that this sensitivity either for the initial conditions
or for the fact if $\frac{1}{Q(x)}dA(x)$ is a positive measure can be used in
the numerical calculation to test if, for example, the roots of the odd
multiplicity of the polynomial $Q(x)$ lie in the support of $dA$ or finding
the first values of the integral (\ref{init}).
\end{remark}

Let us return now to the question if a convolution of two pm sequences is a pm
sequence. From the formula (\ref{av_con}) it follows that the sequence of
arithmetic averages of a convolution of two pm sequences is a pm sequence.
However, if we consider the simple case of one sequence, say, $\left\{
a_{n}\right\}  _{n\geq0}$ is a moments sequence of the distribution $dA$ and
the sequence say $\left\{  b^{n}\right\}  _{n\geq0}$ then the solution of the
difference equation%
\[
r_{n+1}-br_{n}=a_{n},
\]
with $r_{0}\allowbreak=\allowbreak r0$ is given by the formula%
\[
r_{n}\allowbreak=\allowbreak b^{n}r0+\sum_{j=0}^{n-1}b^{n-1-j}a_{j}.
\]
Hence for $r0\allowbreak=\allowbreak0$ sequence $\left\{  r_{n}\right\}  $ is
a sequence of convolutions of $\left\{  b^{n}\right\}  $ and $\left\{
a_{n}\right\}  $. Moreover, we know from the Theorem \ref{main} that it is a
moments sequence if only point $b$ lies outside the support the measure $dA,$
more precisely if $dA(x)/(x-b)$ is a positive measure as it follows from
Theorem \ref{main}.

\begin{remark}
Now, let us notice that the scheme in that we have two positive measures
absolutely continuous with respect to one another is a general situation. One
considers it in the series of papers \cite{Szablowski2010(1)}, \cite{Szab13},
\cite{Szabl22} and \cite{SzabChol}. These papers provide many consequences of
such assumptions, including infinite expansions of the Radon-Nikodym
derivative:
\begin{equation}
\frac{dB}{dA}(x)=\sum_{i\geq0}c_{i}\alpha_{i}(x), \label{Exp}%
\end{equation}
and $\left\{  \alpha_{i}(x)\right\}  _{i\geq0}$ is the sequence of polynomials
orthogonal with respect to the measure $dA.$ The expansion (\ref{Exp})
converges in mean-square $\operatorname{mod}$ $dA$ provided $\int\left(
\frac{dB}{dA}(x)\right)  ^{2}dA(x)<\infty.$ Knowing sequences $\left\{
\alpha_{i}\right\}  $ and $\left\{  \beta_{i}\right\}  $ one is able to find a
numerical sequence $\left\{  c_{i}\right\}  $ and thus get the expansion
(\ref{Exp}). Moreover, following the above mentioned positions of literature,
there exists a finite linear relationship between two sets of polynomials
$\left\{  \alpha_{i}\right\}  $ and $\left\{  \beta_{i}\right\}  $ orthogonal
respectively to $dA$ and $dB$, provided $\frac{dB}{dA}(x)\allowbreak
=\allowbreak\frac{1}{Q(x)}.$ This observation was first made by Pascal Maroni
in a more general but more confining context, not necessarily concerning
measures. Maroni's approach, followed by his associates in the case of
measures concerns mostly polynomials of order at most $4$. For details see
\cite{M1} and \cite{M2} or other,later papers of Maroni et al.. More
precisely, there exists a table $\left\{  w_{i,n}\right\}  _{N\alpha\geq
n\geq0,0\leq i\leq n}$ of real numbers such that:
\[
\alpha_{n}(x)=\sum_{j=0}^{N}w_{j,n}\beta_{n-j}(x),
\]
where $N$ is the order of the polynomial $Q(x).$ This might lead to a new
difference equation, this time with non-constant coefficients.
\end{remark}

\begin{remark}
Note, also, that we could have calculated the moment sequence of $dB$ by
expanding $\frac{1}{Q(x)}$ in an infinite power series and integrating term by
term getting:%
\begin{equation}
\int x^{n}dB(x)\allowbreak=\allowbreak\sum_{j\geq0}d_{j}\int x^{j+n}dA(x),
\label{xp1}%
\end{equation}
where, of course, $\sum_{j\geq0}d_{j}x^{j}$ is the postulated expansion of
$\frac{1}{Q(x)}.$ If the series on the right-hand side of the above-mentioned
formula is convergent, we have the other relationship combining moments of the
measures $dA$ and $dB$. Such an approach of calculating the sequence of
moments of some measure in two different ways can lead to discovering new,
interesting relationships, not only in the moment sequences in question but
also between sequences describing these moment sequences. This was done for
example in the paper \cite{SzabKes}.

As an example, let us consider two distributions : the semicircle with the
density $\frac{1}{2\pi}\sqrt{4-x^{2}}$ for $\left\vert x\right\vert \leq2$ and
zero otherwise and the so-called arcsine distribution with the density
$\frac{1}{\pi\sqrt{4-x^{2}}}$ for $\left\vert x\right\vert <2$ and $0$
otherwise. It is elementary to notice that since these distributions are
symmetric their odd moments are equal to zero. Further, it is well-known that
even elements of the pm sequence generated by these distributions are
respectively the so-called Catalan numbers $C_{n}\allowbreak=\allowbreak
\frac{1}{n+1}\binom{2n}{n}$ and the so-called central binomial coefficients.
Further, we have
\[
\frac{1}{2\pi}\sqrt{4-x^{2}}\times\frac{2}{(4-x^{2})}=\frac{1}{\pi
\sqrt{4-x^{2}}}.
\]
Hence, taking into account that
\[
\frac{2}{(4-x^{2})}=\frac{1}{2}\sum_{j\geq0}x^{2}/4^{j},
\]
for $\left\vert x\right\vert <2,$ the (\ref{xp1}) takes the following form:%
\[
\binom{2n}{n}=\frac{1}{2}\sum_{j\geq0}C_{n+j}/4^{j}.
\]
Note that the convergence here is very slow since $C_{n}\allowbreak
\cong\allowbreak4^{n}/n^{3/2}.$

On the other hand, since we have:%
\[
\frac{1}{2\pi}\sqrt{4-x^{2}}=(2-\frac{1}{2}x^{2})\frac{1}{\pi\sqrt{4-x^{2}}},
\]
we get an obvious relationship:%
\[
2\binom{2n}{n}-\frac{1}{2}\binom{2n+2}{n+1}=C_{n}.
\]
We also see that the sequence $\left\{  \binom{2n}{n}\right\}  _{n\geq0}$ is
the solution of the following difference equation:%
\[
s_{n+1}-4s_{n}=-2C_{n},
\]
with an initial condition : $s_{0}\allowbreak=\allowbreak1.$ Taking into
account general solution of the above-mentioned equation, we end up with the
following identity:%
\[
\binom{2n}{n}=4^{n}-\frac{1}{2}\sum_{j=0}^{n-1}4^{n-j}C_{j}.
\]

\end{remark}

\end{document}